\documentclass[ECP]{ejpecp} 

\usepackage{bbm}
\usepackage{enumitem}

\SHORTTITLE{Asymptotics of correlation function in the long-range random-cluster model}
\TITLE{Sharp asymptotics of correlation functions in the subcritical long-range random-cluster and Potts models}

\AUTHORS{Yacine Aoun\footnote{Université de Genève,
    \EMAIL{yacine.aoun@unige.ch}}}
 

\AMSSUBJ{82B20; 82B43} 
\AMSSUBJSECONDARY{60K35} 

\SUBMITTED{June 27, 2020} 
\ACCEPTED{April 4, 2021} 

\ARXIVID{2007.00116} 

\VOLUME{26}
\YEAR{2021}
\PAPERNUM{22}
\DOI{10.1214/21-ECP390}

\ABSTRACT{For a family of random-cluster models with cluster weights $q\geq 1$, we prove that the probability that $0$ is connected to $x$ is asymptotically equal to $\tfrac{1}{q}\chi(\beta)^{2}\beta J_{0,x}$ for $\beta<\beta_{c}$. The method developed in this article can be applied to any spin model for which there exists a random-cluster representation which is monotonic.}

\begin{document}



\section{Introduction}
\numberwithin{equation}{section}
\subsection{Definitions and main result}
The random cluster-model (also called FK-percolation) was introduced by Fortuin and Kastelyn in 1969 \cite{FK} and has become a fundamental example of dependent percolation, in particular because of its relation to the Potts model. Indeed, the spin correlations of Potts models can be linked to the cluster connectivity properties of their random-cluster representations. This allows the use of geometric techniques developed for percolation to study the Potts model. We refer to \cite{duminilpims,GRIM2} for books on the subject and a recent discussion of existing results.

The model is defined as follows. For a finite subgraph $\Lambda$ of $\mathbb{Z}^{d}$, a percolation configuration $\omega=(\omega)_{x,y\in\Lambda}$ is an element of $\lbrace 0,1\rbrace^{\mathcal{P}_{2}(\Lambda)}$, where $\mathcal{P}_{2}(\Lambda)=\lbrace \lbrace x,y\rbrace : x,y\in\Lambda, x\neq y\rbrace$. A configuration $\omega$ can be seen as a subgraph of $\Lambda$ with vertex-set $\Lambda$ and edge-set given by $\lbrace \lbrace x,y\rbrace\in\mathcal{P}_{2}(\Lambda) : \omega_{x,y}=1\rbrace$. If $\omega_{x,y}=1$, we say that $\lbrace x,y\rbrace$ is open. Let $k(\omega)$ be the number of connected components in $\omega$. 

Consider $J=(J_{x,y})_{x,y\in\Lambda}$ non-negative coupling constants. 
Fix $\beta,q>0$. Let $\mu_{\Lambda,\beta,q}$ be a measure defined for any $\omega\in\lbrace 0,1\rbrace^{\mathcal{P}_{2}(\Lambda)}$ by 
\begin{center}
$\mu_{\Lambda,\beta,q}(\omega)=\dfrac{q^{k(\omega)}}{Z}\prod\limits_{\lbrace x,y\rbrace\in\mathcal{P}_{2}(\Lambda)}(1-e^{-\beta J_{x,y}})^{\omega_{x,y}}$,
\end{center}
where $Z$ is a normalizing constant introduced in such a way that $\mu_{\Lambda,\beta,q}$ is a probability measure. The measure $\mu_{\Lambda,\beta,q}$ is called the random-cluster measure on $\Lambda$ with free boundary conditions. For $q\geq 1$, the measures can be extented to $\mathbb{Z}^{d}$ by taking the weak limit of measures defined in finite volume. 

We say that $x$ and $y$ are connected in $S\subseteq \mathbb{Z}^{d}$ if there exists a finite sequence of vertices $(v_i)_{i=0}^{n}$ in $S$ such that $v_0 =x$, $v_n =y$ and $ \lbrace v_{i} ,v_{i+1}\rbrace$ is \textit{open} for every $0\leq i<n$. We denote this event by $x \overset{S}{\leftrightarrow} y$. If $S=\mathbb{Z}^{d}$, we drop it from the notation. We write $0\leftrightarrow\infty$ if for every $n\in\mathbb{N}$, there exists $x\in\mathbb{Z}^{d}$ such that $0\leftrightarrow x$ and $\vert x\vert\geq n$, where $\vert \cdot\vert$ denotes a norm on $\mathbb{Z}^{d}$.

For $q\geq 1$, the model undergoes a phase transition: there exists $\beta_{c}\in[0,\infty]$ satisfying
\begin{center}
$\mu_{\mathbb{Z}^{d},\beta,q}(0\leftrightarrow \infty)=
\begin{cases} 
=0 & \text{if } \beta<\beta_{c}, \\
>0 &  \text{if } \beta>\beta_{c}. \\
\end{cases}
$
\end{center}
For $\beta<\beta_{c}$, it follows from the definition that $\mu_{\mathbb{Z}^{d},\beta,q}(0\leftrightarrow x)$ goes to $0$ as $\vert x\vert$ goes to infinity. In \cite{duminilcopinOSSS}, it was proved that if the coupling constants are finite-range, meaning that there exists $R>0$ such that $J_{x,y}=0$ whenever $\vert x-y\vert >R$, then the probability of two points being connected decays exponentially fast in distance, i.e. for every $\beta <\beta_{c}$, there exists $c(\beta)>0$ such that for every $x$ in $\mathbb{Z}^{d}$, 
\begin{equation}
\mu_{\mathbb{Z}^{d},\beta ,q}(0\leftrightarrow x)\leq \exp(-c\vert x\vert).
\end{equation}

In this article, we consider the random-cluster models with strictly positive infinite-range coupling constants $(J_{x,y})_{x,y\in\mathbb{Z}^{d}}$ satisfying for every $x,y,z\in\mathbb{Z}^{d}$ 
\begin{enumerate}[label={\ensuremath{\mathbf{H\arabic*}}}, start=1,noitemsep]
\item \label{hyp:1}
There exists $c>0$ such that $J_{0,x}\leq cJ_{0,y}$ if $\vert x\vert \geq\vert y\vert$.
\item  \label{hyp:2}
$J_{x-z,y-z}=J_{x,y}$.
\item  \label{hyp:3}
$\sum\limits_{y\in\mathbb{Z}^{d}}J_{0,y}<\infty$
\item  \label{hyp:4}
For every $x\in\mathbb{Z}^{d}$, for every $\varepsilon>0$, there exists $\delta>0$ such that for every $y\in\mathbb{Z}^{d}$
\begin{center}
$\vert x-y\vert\leq\delta\vert x\vert\qquad \Rightarrow \qquad \vert J_{0,x}-J_{0,y}\vert\leq\varepsilon J_{0,x}$.
\end{center}
\item \label{hyp:5}
There exist $0<\gamma <1, 0<\alpha<1$ and $C_{1}>0$ such that $\sum\limits_{y\in\mathbb{Z}^{d}}(J_{0,y})^{\alpha}<\infty$ and such that for every $x\in\mathbb{Z}^{d}$
\begin{center}
$\log(J_{0,x})^{2}J_{0,u}J_{0,v}\leq C_{1}J_{0,x}(J_{0,v})^{\alpha},$
\end{center}
with $\vert u\vert\geq \vert x\vert /\log(J_{0,x})^{2}$ and $\vert v\vert\geq \vert x\vert^{\gamma}/\log(J_{0,x})^{2}$.
\end{enumerate}
\begin{remark}
The hypothesis~\ref{hyp:5} is a technical one and its meaning will become transparent at the end of the proof of Lemma~\ref{lemma:negligible_terms}.
\end{remark}
\begin{remark}
Important examples of coupling constants satisfying~\ref{hyp:1}-~\ref{hyp:5} are $J_{0,x}=\vert x\vert^{-c}$ with $c>d$, $J_{0,x}=\vert x\vert^{-\log\vert x\vert}$ or more generally $J_{0,x}=e^{-C\log(p(\vert x\vert))^{\gamma}}$ for some polynomial $p\in\mathbb{R}[x]$ of degree at least 1 and $C,\gamma>0$ chosen such that~\ref{hyp:3} holds.
\end{remark}
\begin{remark}\label{Remark:Stretched expo}
The hypothesis~\ref{hyp:5} rules out the stretched exponential decay, i.e. $J_{0,x}=\exp(-\vert x\vert^{\eta})$ with $\eta\in (0,1)$. This implies in particular that 
\begin{equation*}
\lim\limits_{\vert x\vert\rightarrow\infty}\dfrac{\vert x\vert^{\xi}}{-\log(J_{0,x})}=\infty
\end{equation*}
for every $\xi\in (0,1)$.
\end{remark}
We write $o_{x}(1)$ for a function that goes to 0 as $\vert x\vert$ goes to infinity. The main theorem of this article is the following one.
\begin{theorem}\label{thm:main}
If $(J_{x,y})_{x,y\in\mathbb{Z}^{d}}$ satifies~\ref{hyp:1}-~\ref{hyp:5} then for $q\geq 1,\beta<\beta_{c}$ and for every $x\in\mathbb{Z}^{d}$,  
\begin{equation}
\mu_{\mathbb{Z}^{d},\beta,q}(0\leftrightarrow x)=\dfrac{\beta\chi(\beta)^{2}}{q} J_{0,x}(1+o_{x}(1)),
\end{equation}
where $\chi(\beta):=\sum\limits_{x\in\mathbb{Z}^{d}}\mu_{\mathbb{Z}^{d},\beta,q}(0\leftrightarrow x)$.
\end{theorem}
This theorem was already proved for $q=2$ (the Ising model) in \cite{NS}, and a weaker form of this theorem was proved for $q=1$ (Bernoulli percolation) in \cite{BragaLongRange} and for the one-dimensional $O(N)$ models with $1\leq N\leq 4$ in \cite{Spohnlongrange}. They all relied on the Simon-Lieb type inequalities (see \cite{liebsimoninequality}). For $q\notin\lbrace 1,2\rbrace$, the Simon-Lieb inequality is not available, so those approaches cannot be extended. Instead of that, we are going to use the exponential decay of the size of the connected component of $0$ that was recently proved in \cite{H}. The latter used the so-called OSSS inequality introduced in \cite{duminilcopinOSSS}. This inequality was already used to prove sharpness in a lot of models (see \cite{duminilcopinOSSS,duminilcopin2018subcriticalBoolean,MV}) for which there exists a random-cluster type representation which is monotonic (see \cite[Chapter 2]{GRIM2} for a definition of a monotonic measure). Therefore, the OSSS inequality coupled with the approach developed in this article can be applied to study subcritical phases of long-range spin models for which there exists a random-cluster representation which is monotonic (for instance the Ashkin-Teller model, see \cite{PV}).

\subsection{Applications to the ferromagnetic q-state Potts model}
The Potts model is one of the fundamental examples of a lattice spin model undergoing an ordered/disordered phase transition. It generalizes the Ising model by allowing spins to take one of $q$ values, where $q$ is an integer greater than or equal to 2.

The model on $\mathbb{Z}^{d}$ is defined as follows. For a subset $\Lambda$ of $\mathbb{Z}^{d}$, the probability measure is defined for any $\sigma=(\sigma_{x})_{x\in\Lambda} \in\lbrace 1, \dots, q\rbrace^{\Lambda}$ by 
\begin{center}
$\mathbb{P}_{\Lambda,\beta,q}(\sigma):=\dfrac{\exp(-\beta H_{\Lambda,q}(\sigma))}{\sum\limits_{\sigma '\in\lbrace 1,...,q\rbrace^{\Lambda}}\exp(-\beta H_{\Lambda,q}(\sigma '))}$\qquad with\qquad $H_{\Lambda,q}(\sigma):=\sum\limits_{ x,y\in\Lambda}J_{xy}\delta_{\sigma_{x}\neq\sigma_{y}}$.
\end{center} 
The model can be defined on $\mathbb{Z}^{d}$ by taking the weak limit of measures in finite volume. The measure thus obtained is called the measure with free boundary conditions and is denoted by $\mathbb{P}_{\mathbb{Z}^{d},\beta,q}$. The Potts model undergoes a phase transition between the absence and the existence of long-range order at the so-called critical inverse temperature $\beta_{c}$, see \cite{GRIM2} for details. Our main theorem from the point of view of the Potts model is the following one.
\begin{theorem}
If $(J_{x,y})_{x,y\in\mathbb{Z}^{d}}$ satifies~\ref{hyp:1}-~\ref{hyp:5}, then for $q\geq 1,\beta<\beta_{c}$ and $x\in\mathbb{Z}^{d}$,  
\begin{equation}\label{thm:Potts}
\mathbb{P}_{\mathbb{Z}^{d},\beta ,q}(\sigma_{0}=\sigma_{x})-\tfrac{1}{q}=\beta\chi(\beta)^{2}qJ_{0,x}(1+o_{x}(1)),
\end{equation}
where $\chi(\beta):=\tfrac{1}{q-1}\sum\limits_{x\in\mathbb{Z}^{d}}\mathbb{P}_{\mathbb{Z}^{d},\beta ,q}(\sigma_{0}=\sigma_{x})-\frac{1}{q}$.
\end{theorem}
Since the Potts model and the random-cluster models can be coupled (see \cite{GRIM2}) in such a way that 
\begin{center}
$\mathbb{P}_{\mathbb{Z}^{d},\beta,q}(\sigma_{x}=\sigma_{y})-\frac{1}{q}=\frac{q-1}{q}\mu_{\mathbb{Z}^{d},\beta,q}(x\leftrightarrow y)$,
\end{center}
Theorem~\ref{thm:Potts} is a direct consequence of Theorem~\ref{thm:main} and we will therefore focus on Theorem~\ref{thm:main}
\subsection{Background}
The following standard properties will be used in the proof of Theorem~\ref{thm:main}.
\medbreak
\noindent
\textbf{Finite energy property}. For every $\Lambda\subset \mathbb{Z}^{d}, q\geq 1, \omega'\in\lbrace 0,1\rbrace^{\mathcal{P}_{2}(\Lambda)}$ and $x,y\in\Lambda$,
\begin{center}
$\mu_{\Lambda,\beta,q}(\omega_{x,y}=1\vert \omega_{a,b}=\omega'_{a,b}, \forall\lbrace a,b\rbrace\in\mathcal{P}_{2}(\Lambda)\setminus\lbrace x,y\rbrace)\leq\beta J_{x,y}$.
\end{center}
We refer to \cite{duminilpims} for more details about this property.
\medbreak
\noindent
\textbf{Monotonicity of measures}. The following is a standard consequence of the FKG inequality : for $q\geq 1$, two subsets $\Lambda_{1}\subset\Lambda_{2}$ of $\mathbb{Z}^{d}$ and an increasing event $A$ depending on the edges in $\Lambda_{1}$ (see \cite{FV} for definition of an increasing event and the proof of this inequality), we have
\begin{equation}\label{eq:monotonicity}
\mu_{\Lambda_{1},\beta,q}(A)\leq\mu_{\Lambda_{2},\beta,q}(A).
\end{equation}
Finally, the following non-trivial input will be a key ingredient of the proof.
\begin{theorem}\label{thm:expo_decay_volume}
For $q\geq1,\beta<\beta_{c}$, there exists $c_{1}=c_{1}(\beta,q)>0$ such that for every $n \in\mathbb{N}$
\begin{equation}\label{eq:expo_decay_volume}
\mu_{\mathbb{Z}^{d},\beta,q}(\vert C(0)\vert\geq n)\leq\exp(-c_{1}n),
\end{equation}
where $C(0):=\lbrace x\in\mathbb{Z}^{d} : 0\leftrightarrow x\rbrace$.
\end{theorem}
\noindent
This theorem was proved in \cite{H}.
\section{Proof of Theorem 1.4}
\subsection{Upper bound}
\numberwithin{equation}{section}
Fix $(J_{x,y})_{x,y\in\mathbb{Z}^{d}}$ satisfying~\ref{hyp:1}-~\ref{hyp:5}, $\beta<\beta_{c}, q\geq 1$ and $x\in\mathbb{Z}^{d}$. If $0$ is connected to $x$, then there are two possibilites: either there is a big number of 'short' open edges (i.e.\ open edges whose endpoints are close) in $C(0)$ or there is a small number of 'long' open edges in $C(0)$. In the first case, this implies that the number of vertices in $C(0)$ is big, which is unlikely to happen by~\eqref{eq:expo_decay_volume}. In order to make this idea precise, we introduce some notation. From now on, we will write $\mu$ instead of $\mu_{\mathbb{Z}^{d},\beta,q}$. Define $f(x):=-2\log(J_{0,x})/ c_{1}$ where $c_{1}$ is provided by Theorem~\ref{thm:expo_decay_volume}. Denote by $D_{y}$ the event that the size of the connected component of $y$ is smaller than $f(x)$. 
We can partition
\begin{center}
$\mu(0\leftrightarrow x)= \mu(0\leftrightarrow x, D_{0})+\mu(0\leftrightarrow x, D^{c}_{0}).$
\end{center}
Using~\eqref{eq:expo_decay_volume} we easily get that
\begin{center}
$\mu(D_{0}^{c})=o_{x}(1)J_{0,x}$.
\end{center}
This implies that the size of connected component of $0$ can be assumed to be smaller than $f(x)$. In this case, we are going to prove two lemmas. Lemma~\ref{lemma:negligible_terms} gives terms that are negligible with respect to $J_{0,x}$ and Lemma~\ref{lemma:right_asymptotics} gives the sharp asymptotics. 

If $0$ is connected $x$ and the size of the connected component of $x$ is smaller than $f(x)$, then there must exist an open edge in $C(0)$ whose endpoints are separated by a distance at least $\vert x\vert /f(x)$. This will be an important observation in the proof of the next lemma. Before stating the lemma, we introduce some notation. 

For a configuration $\omega$, define the random variable $L_{0,x}(\omega):=\sup\lbrace \vert y_{1}-y_{2}\vert $: $\lbrace y_{1},y_{2}\rbrace$ is open, $0$ is connected to $y_{1}$ and $x$ to $y_{2}$ without using $\lbrace y_{1},y_{2}\rbrace\rbrace$. For $y\in\lbrace 0,x\rbrace$ and $\Lambda\subset\mathbb{Z}^{d}$, define $R^{y}(\Lambda):=\sup\lbrace \vert x-y\vert : x\in\Lambda\cap C(y)\rbrace$. If $\Lambda=C(y)$, we simply write $R^{y}$. If $L_{0,x}<\infty$, then there exists an open edge $\lbrace y_{1},y_{2}\rbrace\in\mathcal{P}_{2}(\mathbb{Z}^{d})$ such that $\vert y_{1}-y_{2}\vert=L_{0,x}$ and $0$ is connected to $y_{1}$ and $x$ to $y_{2}$ without using $\lbrace y_{1},y_{2}\rbrace$. If there are several such edges, take the one that maximizes $R^{y}(C(y)\setminus\lbrace y_{1},y_{2}\rbrace)$. Then, If there are several such edges, define an order $\prec$ on $\mathcal{P}_{2}(\mathbb{Z}^{d})$ and choose the one minimal for $\prec$. In this case, we call $\lbrace y_{1},y_{2}\rbrace$ the \emph{maximal edge} with respect to $y$ and we define $R_{y}:=R^{y}(\tilde{C}(y))$, where $\tilde{C}(y):=\lbrace z\in\mathbb{Z}^{d} : z$ is connected to $y$ without using the edge $\lbrace y_{1},y_{2}\rbrace\rbrace$.
\begin{lemma}\label{lemma:negligible_terms}
For $\beta <\beta_{c},0<\gamma <1$ given by~\ref{hyp:5} and $x\in\mathbb{Z}^{d}$
\begin{equation}\label{eq:negligible_terms}
\mu(0\leftrightarrow x, D_{0}, R_{0}\geq \vert x\vert^{\gamma})=o_{x}(1)J_{0,x}.
\end{equation}
\end{lemma}
\begin{proof}[Proof]
Let $\lbrace y,z\rbrace$ be the maximal edge with respect to $0$. Recall that the edge $\lbrace y,z\rbrace$ is open by definition. If $0$ is connected to $x$, by symmetry, one can assume that $0$ is connected to $y$ and $z$ to $x$ without using the edge $\lbrace y,z\rbrace$. 
If $\vert C(0)\vert\leq f(x)$, then $\vert y-z\vert\geq \vert x\vert/f(x)$.
Set $k(x):=\lfloor \vert x\vert /f(x)\rfloor$ and $E_{y_{1},y_{2}}=\mathcal{P}_{2}(\mathbb{Z}^{d})\setminus\lbrace y_{1},y_{2}\rbrace$. Finally, for $n\in\mathbb{N}$, we set $\Lambda_{n}(y):=\lbrace x\in\mathbb{Z}^{d} : \vert x -y\vert < n\rbrace$. Using the union bound, we get
\begin{align*}
\mu(0\leftrightarrow x, D_{0}, R_{0}\geq \vert x\vert^{\gamma})
&\leq
\sum\limits_{y\in \mathbb{Z}^d}\sum\limits_{z\in \Lambda^{c}_{k}(y)}
\mu(0\overset{E_{y,z}}{\longleftrightarrow}y,\omega_{y,z}=1,z\overset{E_{y,z}}{\longleftrightarrow} x, D_{0}, R_{0}\geq\vert x\vert^{\gamma})
\\
&=
\sum\limits_{y\in \mathbb{Z}^d}\sum\limits_{z\in \Lambda^{c}_{k}(y)}
\mu(0\overset{E_{y,z}}{\longleftrightarrow}y,\omega_{y,z}=1,z\overset{E_{y,z}}{\longleftrightarrow} x, D_{0}, D_{x}, R_{0}\geq\vert x\vert^{\gamma}),
\end{align*}
where the last equality follows from the fact that if $0$ is connected to $x$, then $C(0)=C(x)$.
For $y\in\mathbb{Z}^{d}$, let $\tilde{D}_{y}:=\lbrace \vert\tilde{C}(y)\vert\leq f(x)\rbrace$ where $\tilde{C}(y)$ is defined above.
Notice that the event $D_{y}$ is included in the event $\tilde{D}_{y}$. Using the inclusion of events, conditioning on $\lbrace 0\overset{E_{y,z}}{\longleftrightarrow}y\rbrace\cap\lbrace z\overset{E_{y,z}}{\longleftrightarrow} x\rbrace\cap \tilde{D}_{0}\cap\tilde{D}_{x}\cap\lbrace R_{0}\geq\vert x\vert^{\gamma}\rbrace$ and using the finite energy property, we get

\medskip
$
\begin{aligned}
\mu(0\overset{E_{y,z}}{\longleftrightarrow}y,z\overset{E_{y,z}}{\longleftrightarrow} x, &D_{0}, D_{x}, R_{0}\geq\vert x\vert^{\gamma}, \omega_{y,z}=1)
\leq
\\  
&\leq
\beta J_{y,z}
\mu(0\overset{E_{y,z}}{\longleftrightarrow}y, z\overset{E_{y,z}}{\longleftrightarrow}x, \tilde{D}_{0},\tilde{D}_{x}, R_{0}\geq\vert x\vert^{\gamma})
\\
&\leq c\beta J_{0,k(x)e_{1}}
\mu(0\overset{E_{y,z}}{\longleftrightarrow}y, z\overset{E_{y,z}}{\longleftrightarrow}x, \tilde{D}_{0},\tilde{D}_{x}, R_{0}\geq\vert x\vert^{\gamma}).
\end{aligned}
$

\vspace{4mm}
\noindent
with $c$ given by~\ref{hyp:1}. Plugging this into the inequality above gives 
\begin{align*}
\mu(0\leftrightarrow x, D_{0}, R_{0}\geq \vert x\vert^{\gamma})
&\leq
c\beta J_{0,k(x)e_{1}}\sum\limits_{y\in \mathbb{Z}^d}\sum\limits_{z\in \Lambda^{c}_{k}(y)}
\mu(0\overset{E_{y,z}}{\longleftrightarrow}y, z\overset{E_{y,z}}{\longleftrightarrow}x, \tilde{D}_{0},\tilde{D}_{x}, R_{0}\geq\vert x\vert^{\gamma})
\\
&\leq 
c\beta J_{0,k(x)e_{1}}\sum\limits_{y\in\mathbb{Z}^d}\sum\limits_{z\in\mathbb{Z}^{d}}
\mu(0\overset{E_{y,z}}{\longleftrightarrow}y, z\overset{E_{y,z}}{\longleftrightarrow}x, \tilde{D}_{0},\tilde{D}_{x}, R_{0}\geq\vert x\vert^{\gamma})
\\
&= c\beta J_{0,k(x)e_{1}}\mathbb{E}(\vert \tilde{C}(0)\vert\vert \tilde{C}(x)\vert\mathbbm{1}_{\lbrace \tilde{D}_{0},\tilde{D}_{x}, R_{0}\geq\vert x\vert^{\gamma})}
\\ 
&\leq c\beta f(x)^{2}J_{0,k(x)e_{1}} \mu(\tilde{D}_{0}, R_{0}\geq \vert x\vert^{\gamma}).
\end{align*}
In the last line, we used that $\vert \tilde{C}(0)\vert\leq f(x)$ on $\tilde{D}_{0}$ and $\vert \tilde{C}(x)\vert\leq f(x)$ on $\tilde{D}_{x}$. Observe that if $\vert \tilde{C}(0)\vert\leq f(x)$ and $R_{0}\geq \vert x\vert^{\gamma}$, then there exists $a,b\in\mathbb{Z}^{d}$ such that
\begin{itemize}[noitemsep]
\item  $0$ is connected to $a$ in $\mathcal{P}_{2}(\mathbb{Z}^{d})\setminus\lbrace a,b\rbrace$,
\item $\vert a-b\vert\geq \vert x\vert^{\gamma}/f(x)$,
\item $\lbrace a,b\rbrace$ is open.
\end{itemize}
Using the union bound, we get
\begin{equation*}
\mu(\tilde{D}_{0}, R_{0}\geq\vert x\vert^{\gamma})
\leq
\sum\limits_{a\in \mathbb{Z}^d}\smashoperator[r]{\sum\limits_{\substack{b\in\mathbb{Z}^{d} \\
\vert b-a\vert\geq \vert x\vert^{\gamma}/f(x)} }} 
\mu(0\overset{E_{a,b}}{\longleftrightarrow}a, \omega_{a,b}=1).
\end{equation*}
As before, the conditioning and the finite energy property give
\begin{center}
$\mu(0\overset{E_{a,b}}{\longleftrightarrow}a, \omega_{a,b}=1)
\leq 
\beta J_{a,b}\mu(0\overset{E_{a,b}}{\longleftrightarrow}a).
$
\end{center}
Plugging this into the previous inequality gives
{\allowdisplaybreaks
\begin{align*}
\mu(D_{0}, R_{0}\geq\vert x\vert^{\gamma})
&\leq
\sum\limits_{a\in \mathbb{Z}^d}\smashoperator[r]{\sum\limits_{\substack{b\in\mathbb{Z}^{d} \\
\vert b-a\vert\geq \vert x\vert^{\gamma}/f(x)} }} 
\beta J_{a,b}\mu(0\overset{E_{a,b}}{\longleftrightarrow}a)
\\
&\leq \beta \chi(\beta)\smashoperator{\sum\limits_{\substack{b\in\mathbb{Z}^{d} \\
\vert b\vert\geq\vert x\vert^{\gamma}/f(x)} }}J_{0,b}.
\end{align*}
}
In the second line, we used that $(J_{x,y})_{x,y\in\mathbb{Z}^{d}}$ is invariant under translations. Therefore, combining all the inequalities we get 
\begin{align*}
\mu(0\leftrightarrow x, D_{0}, D_{x}, R_{0}\geq \vert x\vert^{\gamma})
&\leq 
c\chi(\beta)(\beta f(x))^{2}J_{0,k(x)e_{1}}\smashoperator{\sum\limits_{\substack{b\in\mathbb{Z}^{d} \\
\vert b\vert\geq \vert x\vert^{\gamma}/f(x) } }}J_{0,b}
\\
&\leq 
c_{2}J_{0,x}\smashoperator{\sum\limits_{\substack{b\in\mathbb{Z}^{d} \\
\vert b\vert\geq \vert x\vert^{\gamma}/f(x)} }} (J_{0,b})^{\alpha}
=o_{x}(1)J_{0,x}
\end{align*}
with $c_{2}=4C_{1}c^{2}\chi(\beta)\beta^{2}/c_{1}$ where $C_{1}$ is given by ~\ref{hyp:5}. The second inequality follows from the definition of $f(x)$, \ref{hyp:1} and~\ref{hyp:5}. The last equality follows from $\sum_{w\in\mathbb{Z}^{d}}(J_{0,w})^{\alpha}<\infty$ and Remark~\ref{Remark:Stretched expo}. This finishes the proof of Lemma~\ref{lemma:negligible_terms}.
\end{proof}
Lemma~\ref{lemma:negligible_terms} implies by symmetry that
\begin{equation}\label{eq:negligible_other_terms}
\mu(0\leftrightarrow x, D_{0}, R_{x}\geq\vert x\vert^{\gamma})=o_{x}(1)J_{0,x}.
\end{equation}
We can then focus on the next lemma, which gives the sharp asymptotics of the probability of $0$ being connected to $x$.
\begin{lemma}\label{lemma:right_asymptotics}
For $\beta <\beta_{c}, 0<\gamma <1$ given by~\ref{hyp:5} and $x\in\mathbb{Z}^{d}$
\begin{equation}\label{eq:right_asymptotics}
\limsup\limits_{\vert x\vert \rightarrow\infty}\dfrac{\mu(0\leftrightarrow x,  R_{0}\leq\vert x\vert^{\gamma}, R_{x}\leq\vert x\vert^{\gamma}) } {J_{0,x}} \leq \dfrac{\chi(\beta)^{2}\beta}{q}.
\end{equation}
\end{lemma}
The upper bound follows by combining~\eqref{eq:negligible_terms}, \eqref{eq:negligible_other_terms} and~\eqref{eq:right_asymptotics}.
\begin{proof}[Proof]
Set $\Lambda=\Lambda_{\vert x\vert^{\gamma}}(0), \Lambda'=\Lambda_{\vert x\vert^{\gamma}}(x)$ and $P_{0,x}:=\lbrace R_{0}\leq\vert x\vert^{\gamma}\rbrace\cap\lbrace R_{x}\leq\vert x\vert^{\gamma}\rbrace$. Let $\Lambda_{n}$ be such that $\Lambda,\Lambda' \subset \Lambda_{n}$.
Let $\lbrace y,z\rbrace$ be the maximal edge with respect to $0$. If $0$ is connected to $x$ and $R_{x}\leq\vert x\vert^{\gamma}$, then $\lbrace y,z\rbrace$ is also maximal with respect to $x$. By symmetry, one can assume that $0$ is connected to $y$ in $\Lambda$ and $z$ is connected to $x$ in $\Lambda'$. The union bound gives
\begin{align*}
\mu_{\Lambda_n}(0\leftrightarrow x, P_{0,x})
&\leq
\sum\limits_{y\in\Lambda}\sum\limits_{z\in\Lambda'}
\mu_{\Lambda_{n}}(0\overset{\Lambda}{\leftrightarrow}y, \omega_{y,z}=1, z\overset{\Lambda'}{\leftrightarrow}x, P_{0,x})
\\
&=
\sum\limits_{y\in\Lambda}\sum\limits_{z\in\Lambda'}
\frac{1}{q}(1-\exp(-\beta J_{y,z}))
\mu_{\Lambda_{n}}(0\overset{\Lambda}{\leftrightarrow}y, z\overset{\Lambda'}{\leftrightarrow}x, R^{0}\leq\vert x\vert^{\gamma}, R^{x}\leq\vert x\vert^{\gamma})
\\
&\leq
\frac{\beta}{q}\sum\limits_{y\in\Lambda}\sum\limits_{z\in\Lambda'}J_{y,z}
\mu_{\Lambda_{n}}(0\overset{\Lambda}{\leftrightarrow}y, z\overset{\Lambda'}{\leftrightarrow}x, R^{0}\leq\vert x\vert^{\gamma}, R^{x}\leq\vert x\vert^{\gamma}).
\end{align*}
In the second line, we used the fact that on $P_{0,x}$, the number of connected components increases by 1 when $\omega_{y,z}$ goes from $1$ to $0$. In the third line, we used that $1-\exp(-\beta J_{y,z})\leq \beta J_{y,z}$. Fix $\varepsilon >0$. It follows from~\ref{hyp:4} and the translational invariance that 
\begin{center}
$J_{y,z}\leq (1+\varepsilon)J_{0,x},$
\end{center}
since $\vert z-y-x\vert\leq\delta\vert x\vert$ for $\vert x\vert$ big enough. Therefore
\begin{equation*}
\mu_{\Lambda_n}(0\leftrightarrow x, P_{0,x})\leq (1+\varepsilon)J_{0,x}\sum\limits_{y\in\Lambda}\sum\limits_{z\in\Lambda'}
\mu_{\Lambda_{n}}(0\overset{\Lambda}{\leftrightarrow}y, \omega_{y_{1},y_{2}}=0, z\overset{\Lambda'}{\leftrightarrow}x, R^{0}\leq\vert x\vert^{\gamma}, R^{x}\leq\vert x\vert^{\gamma}).
\end{equation*}
Now, we can partition with respect to the possible connected components of $x$ to get 
\begin{equation*}
\mu_{\Lambda_{n}}(0\overset{\Lambda}{\leftrightarrow}y, z\overset{\Lambda'}{\leftrightarrow}x, R^{0}\leq\vert x\vert^{\gamma}, R^{x}\leq\vert x\vert^{\gamma})
=
\sum\limits_{S}\mu_{\Lambda_{n}}(0\overset{\Lambda}{\leftrightarrow}y, R^{0}\leq\vert x\vert^{\gamma}, C(x)=S),
\end{equation*}
where the summation is over $S$ containing $x$ and $z$ such that $R^{x}(S)\leq\vert x\vert^{\gamma}$. Then conditioning on $\lbrace C(x)=S\rbrace$ gives
\begin{align*}
\sum\limits_{S}\mu_{\Lambda_{n}}(0\overset{\Lambda}{\leftrightarrow}y, R^{0}\leq\vert x\vert^{\gamma}, C(x)=S)
&=
\sum\limits_{S}
\mu_{\Lambda_{n}}(0\overset{\Lambda}{\leftrightarrow}y, R^{0}\leq\vert x\vert^{\gamma}\hphantom{,}\vert C(x)=S)\mu_{\Lambda_{n}}(C(x)=S)
\\
&=
\sum\limits_{S}
\mu_{\Lambda_{n}\backslash S}(0\overset{\Lambda}{\leftrightarrow}y, R^{0}\leq\vert x\vert^{\gamma})\mu_{\Lambda_{n}}(C(x)=S)
\\
&\leq
\sum\limits_{S}
\mu_{\Lambda_{n}}(0\overset{\Lambda}{\leftrightarrow}y)\mu_{\Lambda_{n}}(C(x)=S)
\\
&=
\mu_{\Lambda_{n}}(0\overset{\Lambda}{\leftrightarrow}y)\mu_{\Lambda_{n}}(z\overset{\Lambda'}{\leftrightarrow}x, R^{x}\leq\vert x\vert^{\gamma}).
\end{align*}
In the second line, we used the spatial Markov property (see \cite[Chapter~3]{FV}) as well as the fact that if $w\in S$ and $z\notin S$, then $\lbrace w,z\rbrace$ is closed. In the third line, we used the inclusion of events and~\eqref{eq:monotonicity}.
Plugging this into the inequality above and taking the limit as $n$ goes to infinity, we get
\begin{align*}
\mu(0\leftrightarrow x, P_{0,x})
&\leq 
(1+\varepsilon)\frac{\beta}{q}J_{0,x}
\sum\limits_{y\in\Lambda}\sum\limits_{z\in\Lambda'}
\mu(0\overset{\Lambda}{\leftrightarrow}y)\mu(z\overset{\Lambda'}{\leftrightarrow}x, R^{x}\leq\vert x\vert^{\gamma})
\\
&\leq
(1+\varepsilon)\frac{\beta}{q}\chi(\beta)^{2}J_{0,x},
\end{align*}
for $\vert x\vert$ big enough. We used the translational invariance in the second inequality. This finishes the proof of Lemma~\ref{lemma:right_asymptotics}.
\end{proof}
\subsection{Lower bound}
We will use the same argument as in \cite{NS}. In this part, we don't use~\ref{hyp:5}. Set $\delta\in (0,1/2], \Delta_{1}=\Lambda_{\vert x\vert^{\delta}}(0), \Delta_{2}=\Lambda_{\vert x\vert^{\delta}}(x)$ and $\Delta=\Delta_{1}\cup\Delta_{2}$. As we work on $\mathbb{Z}^{d}$, we can take $\delta = 1/2$, but a smaller value may be needed to extend the proof to a different graph and to more general coupling constants. Let $N$ be the number of open edges from $\Delta_{1}$ to $\Delta_{2}$. Then the inclusion of events and the monotonicity of the measure~\eqref{eq:monotonicity} give
\begin{center}
$\mu_{\Delta}(0\leftrightarrow x, N=1)\leq\mu_{\Delta}(0\leftrightarrow x) \leq\mu(0\leftrightarrow x)$.
\end{center}
For $y\in\Delta_{1}, z\in\Delta_{2}$, let $G_{y,z}$ be the event that there exists an unique edge $\lbrace y,z\rbrace$ such that 
\begin{itemize}[noitemsep]
\item $0\overset{\Delta_{1}}{\leftrightarrow}y$,
\item $\lbrace y,z\rbrace$ is open,
\item $z\overset{\Delta_{2}}{\leftrightarrow}x$.
\end{itemize}
In this case, $0$ is connected to $x$ and $N=1$. Therefore
\begin{equation*}
\sum\limits_{y\in\Delta_{1}}\sum\limits_{z\in\Delta_{2}}\mu_{\Delta}(0\overset{\Delta_{1}}{\leftrightarrow}y, \omega_{y,z}=1, z\overset{\Delta_{2}}{\leftrightarrow}x, N=1)
=
\mu_{\Delta}(\bigsqcup\limits_{\substack{y\in\Delta_{1} \\  z\in\Delta_{2}}}G_{y,z})\leq\mu_{\Delta}(0\leftrightarrow x, N=1).
\end{equation*}
In the first equality, we used the fact that the events $G_{y,z}$ are disjoint for different edges. Using the fact that on the event $N=1$,  the number of connected components increases by 1 when $\omega_{y,z}$ goes from 1 to 0, we get that
\begin{equation*}
\mu_{\Delta}(0\overset{\Delta_{1}}{\leftrightarrow}y, \omega_{y,z}=1, z\overset{\Delta_{2}}{\leftrightarrow}x, N=1)
=
\frac{1}{q}(1-\exp(-\beta J_{y,z}))
\mu_{\Delta}(0\overset{\Delta_{1}}{\leftrightarrow}y, z\overset{\Delta_{2}}{\leftrightarrow}x, N=0).
\end{equation*}
Finally, on $N=0$, all the edges between $\Delta_{1}$ and $\Delta_{2}$ are closed, and therefore we can factorize the measure as
\begin{center}
$\mu_{\Delta}(0\overset{\Delta_{1}}{\leftrightarrow}y, z\overset{\Delta_{2}}{\leftrightarrow}x\vert N=0)
=
\mu_{\Delta_{1}}(0\leftrightarrow y)\mu_{\Delta_{2}}(z\leftrightarrow x)$.
\end{center}
Combining all the inequalities we get
\begin{equation}\label{eq:lower_bound_intermediate}
\mu(0\leftrightarrow x)\geq
\sum\limits_{y\in\Delta_{1}}\sum\limits_{z\in\Delta_{2}}
\frac{1}{q}(1-\exp(-\beta J_{y,z}))
\mu_{\Delta_{1}}(0\leftrightarrow y)\mu_{\Lambda'}(z\leftrightarrow x)\mu_{\Delta}(N=0).
\end{equation}
Fix $\varepsilon >0$. It follows from~\ref{hyp:4} and the translational invariance that 
\begin{equation*}
J_{y,z}\geq (1-\varepsilon)J_{0,x},
\end{equation*}
since $\vert z-y-x\vert\leq\delta\vert x\vert$ for $\vert x\vert$ big enough. Therefore, using~\eqref{eq:lower_bound_intermediate}, we get 
\begin{equation}\label{eq:lower_bound_intemediate_2}
\mu(0\leftrightarrow x) \geq
(1-\exp(-\beta(1-\varepsilon)J_{0,x}))\sum\limits_{y\in\Delta_{1}}\sum\limits_{z\in\Delta_{2}}
\frac{1}{q}
\mu_{\Delta_{1}}(0\leftrightarrow y)\mu_{\Delta_{2}}(z\leftrightarrow x)\mu_{\Delta}(N=0).
\end{equation}
By the translational invariance and the monotonicity~\eqref{eq:monotonicity}, we get 
\begin{equation*}
\lim\limits_{\vert x\vert \rightarrow\infty}\sum\limits_{y\in\Delta_{1}}\sum\limits_{z\in\Delta_{2}}\mu_{\Delta_{1}}(0\leftrightarrow y)\mu_{\Delta_{2}}(z\leftrightarrow x)=\chi(\beta)^{2}.
\end{equation*}
Now, let us prove that $\mu_{\Delta}(N=0)$ goes to 1 as $\vert x\vert$ goes to infinity. If $N\geq 1$, then there exist $y,z$ in $\mathbb{Z}^{d}$ such that 
\begin{itemize}[noitemsep]
\item $y\in\Delta_{1},z\in\Delta_{2}$,
\item $\lbrace y,z\rbrace$ is open.
\end{itemize}
Therefore, the union bound gives 
\begin{equation*}
\mu_{\Delta}(N\geq 1) 
\leq
\sum\limits_{y\in\Delta_{1}}\sum\limits_{z\in\Delta_{2}}
\mu_{\Delta}(\omega_{y,z}=1)
\leq
\beta\sum\limits_{y\in\Delta_{1}}\sum\limits_{z\in\Delta_{2}}J_{y,z}\leq (1+\varepsilon)J_{0,x}\vert\Delta_{1}\vert^{2}. 
\end{equation*}
The second inequality follows from the finite energy property and the third inequality from~\ref{hyp:4}. Since $\delta\leq 1/2$ and $\sum_{w\in\mathbb{Z}^{d}}J_{0,w}<\infty$ by~\ref{hyp:3}, it follows that 
\begin{equation}
J_{0,x}\vert\Delta_{1}\vert^{2}=o_{x}(1)
\nonumber
\end{equation}
and therefore $\lim_{\vert x\vert\rightarrow\infty}\mu_{\Delta}(N=0)=1$. The lower bound then follows from~\eqref{eq:lower_bound_intemediate_2} combined with the fact that 
\begin{center}
$\lim\limits_{\vert x\vert \rightarrow \infty}\dfrac{1-\exp(-\beta J_{0,x})}{\beta J_{0,x}}=1$.
\end{center}


\ACKNO{The author would like to warmly thank Hugo Duminil-Copin for his guidance and help through the master thesis as well as reading and pointing out mistakes in the previous versions of the present article. The author would also like to thank Yvan Velenik, Maëllie Godard and two anonymous referees for many helpful comments. Finally, the author would like to thank the Excellence Fellowship program at the University of Geneva for supporting him during his studies.}



\begin{thebibliography}{13}
\bibitem{BragaLongRange} A.G. Braga, L.M. Ciolleti, and R. Sanchis. Decay Properties of the Connectivity for Mixed Long Range Percolation Models on $\mathbb{Z}^{d}$, \emph{Journal of Statistical Physics}, 129:587–591, 2007.
\bibitem{duminilpims} H. Duminil-Copin. Lectures on the Ising and Potts models on the hypercubic lattice. 2017. arXiv: 1707.00520.
\bibitem{duminilcopinOSSS} H. Duminil-Copin, A. Raoufi, and V. Tassion. Sharp phase transition for the random-cluster and Potts models
via decision trees. \emph{Annals of Mathematics}, 189(1):75–99, 2019.
\bibitem{duminilcopin2018subcriticalBoolean} H. Duminil-Copin, A. Raoufi, and V. Tassion. Subcritical phase of d-dimensional Poisson-Boolean
percolation and its vacant set, 2018. arXiv: 1805.00695
\bibitem{FK} C.M. Fortuin and P.W. Kastelyn. On the random-cluster model. I. Introduction and relation to other models.
\emph{Physica}, 57(4):536–564, 1972.
\bibitem{FV}S. Friedli and Y. Velenik. \emph{Statistical Mechanics of Lattice Systems: a Concrete Mathematical Introduction}. Cambridge University Press, 2017.
\bibitem{GRIM2} G. Grimmett. \emph{The random-cluster model}. Springer-Verlag, 2006.
\bibitem{H} T. Hutchcroft. New critical exponent inequalities for percolation and the random cluster model. 2019. arXiv: 1901.10363.
\bibitem{liebsimoninequality} E. Lieb. A refinement of Simon’s correlation inequality. \emph{Communications in Mathematical Physics}, 77(2):127– 135, 1980.
\bibitem{MV} S. Muirhead and H. Vanneuville. The sharp phase transition for level set percolation of smooth planar Gaussian fields. \emph{Ann. Inst. H. Poincare Probab. Statist.}, 56:1358–1390, 2020.
\bibitem{NS} C. Newman and H. Spohn. The Shiba relation for the spin-boson model and asymptotic decay in ferromagnetic Ising models. Unpublished, 1998.
\bibitem{PV} C.-E. Pfister and Y. Velenik. Random-cluster representation of the Ashkin-Teller model. \emph{Journal of Statistical Physics}, 88:1295–1331, 1997.
\bibitem{Spohnlongrange} H. Spohn and W. Zwerger. Decay of the Two-Point Function in One-Dimensional $O(N)$ Spin Models with Long-Range Interactions. Journal of Statistical Physics 94, 1037–1043, 1999. 
\end{thebibliography}
\end{document}